\documentclass[leqno]{amsart}

\usepackage{amssymb, amsmath, amsthm, chngcntr, enumerate, mathabx, mathrsfs, mathtools}
\usepackage[numbers]{natbib}

\numberwithin{equation}{section}

\newtheorem{thm}{Theorem}
\newtheorem{theorem}[thm]{Theorem}

\newtheorem{lemma}[thm]{Lemma}

\theoremstyle{remark}

\title[Sharp estimates for a class of Littlewood-Paley operators]{Sharp asymptotic estimates for a class of Littlewood-Paley operators}
\author{Odysseas Bakas}
\address{Centre for Mathematical Sciences, Lund University, 221 00 Lund, Sweden}
\email{odysseas.bakas@math.lu.se}

 
\begin{document}

\begin{abstract} It is well-known that Littlewood-Paley operators formed with respect to lacunary sets of finite order are bounded on $L^p (\mathbb{R})$ for all $1<p<\infty$. In this note it is shown that
$$  \| S_{\mathcal{I}_{E_2}} \|_{L^p (\mathbb{R}) \rightarrow L^p (\mathbb{R})} \sim (p-1)^{-2} \quad (p \rightarrow 1^+) ,$$
where $S_{\mathcal{I}_{E_2}}$ denotes the classical Littlewood-Paley operator formed with respect to the second order lacunary set $ E_2 = \{ \pm  ( 2^k - 2^l ) : k,l \in \mathbb{Z} \text{ with } k>l  \}  $.  
\end{abstract}

\maketitle


\section{Introduction}\label{intro}

If $\mathcal{I}$ is a collection of mutually disjoint intervals in the real line, then the corresponding Littlewood-Paley operator $S_{\mathcal{I}}$ is given by
$$ S_{\mathcal{I}} (f) : = \Big( \sum_{I \in \mathcal{I} } | P_I (f) |^2 \Big)^{1/2} . $$ 
Here $P_I$ denotes the Fourier multiplier operator with symbol $\chi_I$. 
It is well-known that if we consider the collection $\mathcal{I}_{E_1} : = \{ [2^j , 2^{j+1}) \}_{j \in \mathbb{Z}} \cup \{ [- 2^{j+1} , - 2^j ) \}_{j \in \mathbb{Z}}$, then $S_{\mathcal{I}_{E_1}}$ is an $L^p$-bounded operator for all $p \in (1, \infty)$ and moreover, for each $p \in (1, \infty)$ there exist positive constants $A_{\mathcal{I}_{E_1}, p}$ and $B_{\mathcal{I}_{E_1}, p}$ such that
\begin{equation}\label{LP_0}
A_{\mathcal{I}_{E_1}, p} \| f \|_{L^p (\mathbb{R})} \leq \| S_{\mathcal{I}_{E_1}} (f) \|_{L^p (\mathbb{R})} \leq B_{\mathcal{I}_{E_1}, p} \| f \|_{L^p (\mathbb{R})} \quad (1 < p < \infty). 
\end{equation}
See e.g. Chapter IV in \cite{Singular_integrals}. In \cite{Carleson}, L. Carleson proved that if we consider the collection $\mathcal{I}_0 = \{ [j-1,j) \}_{j \in \mathbb{N}} \cup \{ [-j , -j +1 ) \}_{j \in \mathbb{N}}$, then $S_{\mathcal{I}_0}$ is $L^p$-bounded for $p \in [2, \infty)$ but is {\it not} bounded on $L^p (\mathbb{R})$ when $p \in (1,2)$. A different proof of Carleson's result was given by A. C\'ordoba in \cite{C}. In \cite{RdF}, J. L. Rubio de Francia extended the aforementioned result by showing that for any collection $\mathcal{I}$ in $\mathbb{R}$ of mutually disjoint intervals the corresponding Littlewood-Paley operator  $S_{\mathcal{I}}$ is bounded on $L^p (\mathbb{R})$ for all $p \in [2, \infty)$. For alternative proofs and extensions of Rubio de Francia's theorem in the one-dimensional case, see J. Bourgain \cite{Bourgain_85}, P. Sj\"olin \cite{Sjolin}, and S. V. Kislyakov and D. V. Parilov \cite{KP}. To the best of our knowledge, the problem of characterising all admissible collections $\mathcal{I}$ in $\mathbb{R}$ of mutually disjoint intervals such that $S_{\mathcal{I}}$ is $L^p$-bounded for all $p \in (1, \infty)$ seems to be still open; see the paper \cite{HK_95} of K. E. Hare and I. Klemes for a relevant conjecture as well as for some related partial results in the periodic setting. See also \cite{HK_92} and \cite{HK_89}. For more details on topics related to Rubio de Francia's theorem, see M. Lacey's paper \cite{Lacey} and the references therein. 

In 1939, J. Marcinkiewicz in his classical paper \cite{M} showed that the periodic Littlewood-Paley operator formed with respect to the second order lacunary set $\{ \pm ( 2^k + 2^l ) : k,l \in \mathbb{N}_0 \text{ with } k > l \}$ is bounded on $L^p (\mathbb{T})$ for all $p \in (1, \infty)$; see \cite[Th\'eor\`eme 8]{M}. In 1981, in \cite{Sj_Sj}, P. Sj\"ogren and Sj\"olin developed a systematic method of successively constructing collections $\mathcal{I}$ of intervals in $\mathbb{R}$ such that the corresponding Littlewood-Paley operator $S_{\mathcal{I}}$ is bounded on $L^p (\mathbb{R})$ for all $p \in (1, \infty)$. To be more specific, following \cite{Sj_Sj}, if $E$, $E'$ are two closed null sets in $\mathbb{R}$, then $E'$ is said to be a {\it successor} of $E$ whenever there exists a constant $c_{E,E'}>0$ such that for every $x,y \in E'$ with $x \neq y$ one has $|x-y| \geq c_{E,E'} \mathrm{dist} (x, E)$. It was shown by Sj\"ogren and Sj\"olin that if a set $E'$ is a successor of a closed null set $E$ and $S_{\mathcal{I}_E}$ is $L^p$-bounded for some $p \in (1, \infty)$, then $S_{\mathcal{I}_{E'}}$ is also bounded on $L^p (\mathbb{R})$; see \cite[Theorem 1.2]{Sj_Sj}. In particular, since the second order lacunary set
$ E_2 : = \{ \pm ( 2^k - 2^l ) : k,l \in \mathbb{Z}\ \text{with} \ k>l \}$ is a successor of the lacunary set $E_1 : = \{ \pm 2^j \}_{j \in \mathbb{Z}} $ and $S_{\mathcal{I}_{E_1}} $ satisfies \eqref{LP_0}, one deduces that for each $p \in (1,\infty)$ there exist positive constants $A_{\mathcal{I}_{E_2}, p}$ and $B_{\mathcal{I}_{E_2}, p}$ such that 
\begin{equation}\label{LP_I_2}
A_{\mathcal{I}_{E_2}, p} \| f \|_{L^p (\mathbb{R})} \leq \| S_{\mathcal{I}_{E_2}} (f) \|_{L^p (\mathbb{R})} \leq B_{\mathcal{I}_{E_2}, p} \| f \|_{L^p (\mathbb{R})} \quad (1 <  p < \infty). 
\end{equation}
Here we adopt the following convention: if $K$ is an infinite countable set in $\mathbb{R}$ that can be written as $K = \{ a_n \}_{n \in \mathbb{Z}}$ in $\mathbb{R}$ with $a_n < a_{n+1}$ for all $n \in \mathbb{Z}$ and moreover, $ a_n \rightarrow - \infty$ as $n \rightarrow - \infty$ and $a_n \rightarrow + \infty$ as $n \rightarrow + \infty$, then $\mathcal{I}_K$ denotes the collection $\{ [ a_n, a_{n+1}) : n \in \mathbb{Z}\}$.

The main goal of this note is to determine the behaviour of the best constant $B_{\mathcal{I}_{E_2}, p}$ in \eqref{LP_I_2} `near' $p=1$ or, in other words, to establish sharp asymptotic estimates for the $L^p-L^p$ operator norm of $S_{\mathcal{I}_{E_2}}$ as $p \rightarrow 1^+$. Before stating our results, let us mention that it follows from the work of Bourgain \cite{Bourgain_89} that in the lacunary case $E_1 = \{ \pm 2^j \}_{j \in \mathbb{Z}} $ the best constant $B_{\mathcal{I}_{E_1}, p} $ in the classical Littlewood-Paley inequality \eqref{LP_0} behaves like $(p-1)^{-3/2}$ as $p \rightarrow 1^+$, namely there exist absolute constants $c_1,c_2 >0$ such that
\begin{equation}\label{Bourgain_3/2}
\frac{c_1}{(p-1)^{3/2}} \leq  \| S_{\mathcal{I}_{E_1}} \|_{L^p (\mathbb{R}) \rightarrow L^p (\mathbb{R})} \leq \frac{c_2}{(p-1)^{3/2}}  \quad (1 < p \leq 2).
\end{equation}
To be more precise, in \cite{Bourgain_89}, Bourgain established a periodic version of \eqref{Bourgain_3/2} and his proof was obtained in \cite{Bourgain_89} by using a classical inequality due to S.-Y. A. Chang, J. M. Wilson, and T. H. Wolff on dyadic martingales \cite{CWW} combined with explicit formulas for translations of dyadic systems \cite{Bourgain_86} and appropriate vector-valued inequalities.  
An alternative proof of the upper estimate in \eqref{Bourgain_3/2} in the periodic setting was given by the author in \cite{Bakas_19} by using the work of T. Tao and J. Wright on Marcinkiewicz multiplier operators \cite{TW} and Tao's converse extrapolation theorem \cite{Tao}. Recently, in \cite{Lerner}, A. K. Lerner showed that 
\begin{equation}\label{Lerner_bound}
\| S_{\mathcal{I}_{E_1}} \|_{L^2 (w) \rightarrow L^2 (w)} \lesssim [w]_{A_2}^{3/2}
\end{equation}
for any $A_2$ weight $w$ on the real line; see \cite[Theorem 1.1]{Lerner}. Moreover, as explained in \cite{Lerner}, by combining \eqref{Lerner_bound} with an extrapolation result due to J. Duoandikoetxea \cite{Duoandikoetxea} one obtains yet another proof of \eqref{Bourgain_3/2}; see \cite[Remark 4.2]{Lerner}. The main ingredients in Lerner's proof of \eqref{Lerner_bound} in \cite{Lerner} were an appropriate variant of the Chang-Wilson-Wolff inequality \cite{CWW} (see \cite[Theorem 2.7]{Lerner}) as well as sharp weighted estimates for multiplier operators of the form $P_I T_m $ with $T_m$ being a multiplier operator whose symbol $m$ satisfies a Marcinkiewicz-type condition on $I$; see \cite[Lemma 3.2]{Lerner}. For sparse bounds for Rubio de Francia-type operators $S_{\mathcal{I}}$ and for Marcinkiewicz-type multiplier operators in the Walsh-Fourier setting, we refer the reader to the recent papers \cite{GRS} and \cite{CCDLO}, respectively. 

As mentioned above, this note focuses on the behaviour of the $L^p-L^p$ operator norm of $S_{\mathcal{I}_{E_2}}$ `near' $p=1^+$. More specifically, our main result is the following sharp asymptotic estimate as $p \rightarrow 1^+$. 

\begin{theorem}\label{p-p}
There exist absolute constants $c_1, c_2 >0$ such that
$$ \frac{c_1} { (p-1)^2 } \leq \| S_{\mathcal{I}_{E_2}} \|_{L^p (\mathbb{R}) \rightarrow L^p (\mathbb{R})} \leq \frac{ c_2 } { (p-1)^2 }  $$
for all $1 < p \leq 2$.
\end{theorem}

The lower estimate in Theorem \ref{p-p} is obtained by adapting the corresponding argument of Bourgain  that establishes the lower estimate in \cite[Theorem 1]{Bourgain_89}. The upper estimate in Theorem \ref{p-p} is a consequence of the following result. 

\begin{theorem}\label{main}
There exists an absolute constant $c_0 > 0$ such that 
$$ \| S_{\mathcal{I}_{E_2}} \|_{L^2 (w) \rightarrow L^2 (w)} \leq c_0 [w]_{A_2}^2 $$
for all $A_2$ weights $w$ on the real line.
\end{theorem}

Notice that the lower estimate in Theorem \ref{p-p} shows that the exponent $r=2$ in $[w]_{A_2}^r$ in Theorem \ref{main} is best possible. At this point, we remark that by carefully examining the proof of \cite[Theorem 1.2]{Sj_Sj} for the case of $S_{\mathcal{I}_{E_2}}$ (and by invoking results of Tao and Wright on Marcinkiewicz multiplier operators \cite{TW}) one gets $ \| S_{\mathcal{I}_{E_2}} \|_{L^p (\mathbb{R}) \rightarrow L^p (\mathbb{R})} \lesssim (p-1)^{-5/2} $ for $p$ `close' to $1^+$. 

Theorem \ref{main} is established by suitably modifying Lerner's proof of \eqref{Lerner_bound} and the `scheme' of its proof is roughly as follows: after reformulating the problem by using duality, one employs Lerner's variant of the Chang-Wilson-Wolff inequality \cite[Theorem 2.7]{Lerner}. Then, the idea is to perform suitable `frequency translations' of the Littlewood-Paley projections appearing in the definition of $S_{\mathcal{I}_{E_2}}$ so that one can effectively use \cite[Theorem 2.7]{Lerner} again. After doing so, one completes the proof by using an appropriate modification of  \cite[Lemma 3.2]{Lerner}; see Lemma \ref{modification} below. The details of the proof of Theorem \ref{proof_2} are presented in Section \ref{proof_2}. In Section \ref{proof_1} we give the proof of Theorem \ref{p-p} and in Section \ref{remarks} we make some further remarks related to the present work and more specifically, we present periodic versions of Theorems \ref{p-p} and \ref{main} and extensions for certain lacunary sets of order $N \in \mathbb{N}$, $N \geq 2$.


\subsection*{Notation}
If $\alpha \in \mathbb{C}$, then $|\alpha|$ stands for the modulus of the complex number $\alpha$, whereas if $A \subseteq \mathbb{R}$ is measurable, $|A|$ denotes the Lebesgue measure of $A$. 

Given two positive quantities $A$ and $B$, if there exists an absolute constant $c_0>0$ such that $A \leq c_0 B$, we shall write $A \lesssim B$ or $B \gtrsim A$. If $A \lesssim B$ and $B \lesssim A$, we write $A \sim B$.  

As usual, the class of Schwartz functions on the real line is denoted by $\mathcal{S} (\mathbb{R})$ and $ C^{\infty}_c (\mathbb{R})$ stands for the class of $C^{\infty}$-functions with compact support in $\mathbb{R}$. 

If $g$ is an integrable function on the torus $\mathbb{T} : = \mathbb{R}/ \mathbb{Z}$, then its Fourier coefficient at $n \in \mathbb{Z}$ is given by $\widehat{g} (n) : = \int_{\mathbb{T}} g (\theta) e^{-i 2 \pi n \theta} d \theta$. 
If $f \in \mathcal{S} (\mathbb{R})$, then its Fourier transform is given by $\widehat{f} (\xi) : = \int_{\mathbb{R}} f(x) e^{- i 2 \pi \xi x} dx$, $\xi \in \mathbb{R}$. 
If $m \in L^{\infty} (\mathbb{R})$, then $T_m$ denotes the multiplier operator with symbol $m$, namely $ \widehat{T_m (f)} (\xi) = m (\xi) \widehat{f} (\xi)$, $\xi \in \mathbb{R}$ for $f \in \mathcal{S} (\mathbb{R})$. If $I \subseteq \mathbb{R}$, we write $P_I = T_{\chi_I}$.  

If $g$ is a locally integrable function on the real line and $I \subset  \mathbb{R}$ is an interval, we use the standard notation
$$ \langle g \rangle_I : = | I |^{-1} \int_I g(x) dx. $$
A non-negative, locally integrable function $w$ on $\mathbb{R}$ is said to be an $A_2 $ weight if, and only if,
$$ [w]_{A_2} := \sup_{ \substack{ I \subset \mathbb{R} : \\ \text{interval}}}  \langle w \rangle_I   \langle w^{-1} \rangle_I < \infty . $$

If $f$ is a locally integrable function on $\mathbb{R}$ then its non-centred Hardy-Littlewood maximal function $M(f)$ is given by
$$ M (f) (x) : = \sup_{  \substack{ I \ \text{interval}: \\ x \in I }} | I |^{-1} \int_I |f(y)| dy  \quad (x \in \mathbb{R}).$$
The maximal Hilbert transform $H^{\ast} (g)$ of $g \in \mathcal{S} (\mathbb{R})$ is given by
$$ H^{\ast} (g) (x) : = \sup_{\epsilon > 0} \Bigg| \int_{ |x - y| > \epsilon } \frac{g(y)}{x-y} dy \Bigg| \quad (x \in \mathbb{R}). $$ 


\section{Proof of Theorem \ref{main}}\label{proof_2}

Arguing as in \cite{Lerner}, it follows from duality that, to prove Theorem \ref{main}, it suffices to show that there exists an absolute constant $A_0 > 0 $ such that
\begin{equation}\label{dual_main+}
\Bigg\| \sum_{k \in \mathbb{Z}} \sum_{ \substack{ l \in \mathbb{Z} : \\ l < k }} P_{I_{k,l}^+} ( \psi_{k,l} ) \Bigg\|_{L^2 (\sigma)} \leq A_0 [\sigma]_{A_2}^2 \Bigg\| \Bigg( \sum_{k \in \mathbb{Z}} \sum_{ \substack{ l \in \mathbb{Z} : \\ l < k }} | \psi_{k,l} |^2 \Bigg)^{1/2} \Bigg\|_{L^2 (\sigma)}
\end{equation}
and
\begin{equation}\label{dual_main-}
\Bigg\| \sum_{k \in \mathbb{Z}} \sum_{ \substack{ l \in \mathbb{Z} : \\ l < k }} P_{I_{k,l}^-} ( \psi_{k,l} ) \Bigg\|_{L^2 (\sigma)} \leq A_0 [\sigma]_{A_2}^2 \Bigg\| \Bigg( \sum_{k \in \mathbb{Z}} \sum_{ \substack{ l \in \mathbb{Z} : \\ l < k }} | \psi_{k,l} |^2 \Bigg)^{1/2} \Bigg\|_{L^2 (\sigma)}
\end{equation}
for all $A_2$ weights $\sigma$ on $\mathbb{R}$ and for any collection of Schwartz functions $ \{ \psi_{k,l} \}_{k > l} $, where only finitely many of the functions $ \psi_{k,l}$ are non-zero. Here for $k,l \in \mathbb{Z}$ with $k>l$, we use the notation $I_{k,l}^+ := [ 2^k -2^l, 2^k -2^{l-1} )$ and $I_{k,l}^- := [- 2^k + 2^{l-1}, -2^k + 2^l )$.

By using \cite[Theorem 2.7]{Lerner}, one deduces that there exist dyadic lattices $\mathcal{I}_j$, $j \in \{1,2,3 \}$, and a $\phi \in C^{\infty}_c (\mathbb{R})$ with $ \mathrm{supp} (\phi) \subseteq [-2,-1/2] \cup [1/2,2]$ so that
\begin{equation}\label{rev_1}
\Bigg\| \sum_{k \in \mathbb{Z}} \sum_{ \substack{ l \in \mathbb{Z} : \\ l < k }} P_{I_{k,l}^+} ( \psi_{k,l} ) \Bigg\|_{L^2 (\sigma)} \lesssim [\sigma]_{A_2}^{1/2} \sum_{j=1}^3  \Bigg\| S_{\phi, \mathcal{I}_j} \Bigg( \sum_{k \in \mathbb{Z}} \sum_{ \substack{ l \in \mathbb{Z} : \\ l < k }} P_{I_{k,l}^+} ( \psi_{k,l} ) \Bigg) \Bigg\|_{L^2 (\sigma)},
\end{equation}
where
$$ S_{\phi, \mathcal{I}_j} (g) (x) : = \Bigg( \sum_{\nu \in \mathbb{Z} } \sum_{\substack{ I \in \mathcal{I}_j : \\ | I | =2^{- \nu} } } \Big( \frac{1}{ | I | } \int_I | \widetilde{P}_{\nu} (g) (x') |^2 d x' \Big) \chi_I (x) \Bigg)^{1/2} $$
and $ \widetilde{P}_{\nu} $ denotes the multiplier operator with symbol $ \phi (2^{-\nu} \xi)$, $\xi \in \mathbb{R}$. For the definition of dyadic lattices and their basic properties, we refer the reader to \cite{LN}. Notice that
\begin{equation}\label{observation}
S_{\phi, \mathcal{I}_j} \Bigg( \sum_{k \in \mathbb{Z}} \sum_{ \substack{ l \in \mathbb{Z} : \\ l < k }}  P_{I_{k,l}^+} ( \psi_{k,l} ) \Bigg) (x) \leq   \Gamma_{\mathcal{I}_j,0} (x) +  \Gamma_{\mathcal{I}_j,1} (x) \quad \text{for} \ j \in \{ 1,2,3\},
\end{equation}
where 
$$ \Gamma_{\mathcal{I}_j,r} (x) : = \Bigg( \sum_{\nu \in \mathbb{Z}} \sum_{\substack{ I \in \mathcal{I}_j : \\ | I | =2^{- \nu} } }  \Bigg[ \frac{1}{| I |} \int_I \Bigg| \widetilde{P}_{\nu} \Bigg( \sum_{l < \nu + r} P_{I_{\nu+r, l}^+} (\psi_{\nu+r, l}) \Bigg) (y) \Bigg|^2 d y \Bigg] \chi_I (x) \Bigg)^{1/2} , $$
$r \in \{ 0,1 \}$. In view of \eqref{rev_1}, to prove \eqref{dual_main+} it suffices to show that
\begin{equation}\label{bound_j,r}
\| \Gamma_{\mathcal{I}_j,r} \|_{L^2 (\sigma)} \lesssim [ \sigma ]_{A_2}^{3/2} \Bigg\| \Bigg( \sum_{k \in \mathbb{Z}} \sum_{\substack{ l \in \mathbb{Z} : \\ l < k }} | \psi_{k,l} |^2 \Bigg)^{1/2} \Bigg\|_{L^2 (\sigma)} 
\end{equation}
for $ j \in \{1,2,3\}$ and $ r \in \{0,1\}$.

Fix a $j \in \{ 1, 2, 3 \}$. We shall only focus on the proof of \eqref{bound_j,r} for $r=0$, as the other case is treated similarly. Observe that one can write  
$$ \Bigg| \widetilde{P}_{\nu} \Bigg( \sum_{ l < \nu} P_{I_{\nu, l}^+} (\psi_{\nu, l}) \Bigg) \Bigg| = \Bigg| \sum_{ l < \nu } P_{I_l^-} \big( \rho_{\nu, l} \big) \Bigg|, $$
where $I_l^- := [ -2^l , -2^{l-1} )$, $\rho_{\nu, l} : = \widecheck{\theta_{\nu}} \ast \widetilde{\psi}_{\nu, l} $ with $\theta_{\nu} (\xi) := \phi (2^{-\nu} (\xi+2^{\nu}))$, $\xi \in \mathbb{R}$, and $\widetilde{\psi}_{\nu, l} (x) := e^{-i 2 \pi 2^{\nu} x} \psi_{ \nu, l } (x) $, $x \in \mathbb{R}$. Hence, to prove \eqref{bound_j,r} for $r=0$, it suffices to show that
\begin{equation}\label{bound_j,0}
\sum_{ \nu \in \mathbb{Z}} \Big\| \sum_{l < \nu} P_{I_l^-} \big( \rho_{\nu, l} \big) \Big\|_{L^2 (\sigma_{\nu, \mathcal{I}_j}) }^2 \lesssim [ \sigma ]_{A_2}^3 \Bigg\| \Bigg( \sum_{k \in \mathbb{Z}} \sum_{ \substack{ l \in \mathbb{Z} : \\ l < k }} | \psi_{k,l} |^2 \Bigg)^{1/2} \Bigg\|_{L^2 (\sigma)}^2  
\end{equation}
for $j \in \{1,2,3\}$, where
$$ \sigma_{\nu, \mathcal{I}_j} (x) : = \sum_{\substack{I \in \mathcal{I}_j : \\ | I | = 2^{-\nu} } } \langle \sigma \rangle_I \chi_I (x). $$
To prove \eqref{bound_j,0}, fix a $\nu \in \mathbb{Z}$ and note that by employing \cite[Lemma 3.1]{Lerner} and \cite[Theorem 2.7]{Lerner} one deduces that 
\begin{equation}\label{bound_nu}
\Bigg\| \sum_{l < \nu} P_{I_l^-} \big( \rho_{\nu, l} \big) \Bigg\|_{L^2 (\sigma_{\nu, \mathcal{I}_j}) } \lesssim 
 [\sigma ]_{A_2}^{1/2} \sum_{j'=1}^3 \Bigg\| S_{\phi, \mathcal{I}_{j'}} \Big( \sum_{l < \nu} P_{I_l^-} \big( \rho_{\nu, l} \big) \Big) \Bigg\|_{L^2 (\sigma_{\nu, \mathcal{I}_j})}
\end{equation}
for $j' \in \{ 1,2,3 \}$. Observe that since $\mathrm{supp} (\phi) \subseteq [-2,-1/2] \cup [1/2,2]$, one has
\begin{equation}\label{pw_ineq}
S_{\phi, \mathcal{I}_{j'}} \Big( \sum_{l < \nu} P_{I_l^-} \big( \rho_{\nu, l} \big) \Big) \leq \sum_{r' \in \{-1, 0, 1\} } K_{\mathcal{I}_{j'}, \nu, r'} , 
\end{equation}
where for $r' \in \{0, 1 \}$ we have
$$ K_{\mathcal{I}_{j'}, \nu, r'} (x) : =  
\Bigg( \sum_{\mu < \nu - 1 } \sum_{\substack{J \in \mathcal{I}_{j'} : \\ |J| =2^{-\mu} }} \Bigg[ \frac{1}{|J|} \int_J \Big| \widetilde{P}_{\mu} \Big( P_{I^-_{\mu+r'}} \big( \rho_{\nu, \mu +r'} \big) \Big) (y) \Big|^2 dy \Bigg] \chi_J (x) \Bigg)^{1/2}  $$
and for $r' = - 1$,
$$ K_{\mathcal{I}_{j'}, \nu, -1} (x) : =  
\Bigg(  \sum_{\substack{J \in \mathcal{I}_{j'} : \\ |J| =2^{-( \nu - 1) } }} \Bigg[ \frac{1}{|J|} \int_J \Big| \widetilde{P}_{\nu - 1} \Big( P_{I^-_{\nu - 1}} \big( \rho_{\nu, \nu - 1} \big) \Big) (y) \Big|^2 dy \Bigg] \chi_J (x) \Bigg)^{1/2} .  $$
We shall prove that there exists an absolute constant $C_0 > 0$ such that for each $\nu \in \mathbb{Z}$ one has 
\begin{equation}\label{bound_K_{0,1}}
\big\| K_{\mathcal{I}_{j'}, \nu, r' } \big\|^2_{L^2 (\sigma_{\nu, \mathcal{I}_j})} \leq C_0 [\sigma]^2_{A_2} \Big\|   \Big( \sum_{l < \nu - 1 } | \psi_{\nu, l}|^2 \Big)^{1/2} \Big\|^2_{L^2 (\sigma )} \quad \text{for}\ r' \in \{0,1\}
\end{equation}
and
\begin{equation}\label{bound_K_{-1}}
\big\| K_{\mathcal{I}_{j'}, \nu, -1 } \big\|^2_{L^2 (\sigma_{\nu, \mathcal{I}_j})} \leq C_0 [\sigma]^2_{A_2} \| \psi_{\nu, \nu- 1} \|^2_{L^2 (\sigma )}
\end{equation}
for all $j' \in \{1,2,3\}$. We shall only provide the details of the proof of \eqref{bound_K_{0,1}} for $r'=0$, as the other cases, i.e. \eqref{bound_K_{0,1}} for $r'=1$ and \eqref{bound_K_{-1}} are treated similarly. To prove \eqref{bound_K_{0,1}} for $r' =0$, fix a $j' \in \{1,2,3\}$ and write  
\begin{equation}\label{equality}
\big\| K_{\mathcal{I}_{j'}, \nu, 0} \big\|^2_{L^2 (\sigma_{\nu, \mathcal{I}_j}) } = \sum_{ \mu < \nu - 1} \big\|   T_{ m_{ \nu, \mu }} \big( P_{I_{\mu}^-} ( \widetilde{\psi}_{\nu, \mu} ) \big) \big\|_{L^2 ( (\sigma_{\nu, \mathcal{I}_j})_{\mu, \mathcal{I}_{j'}}) }^2 
\end{equation}
where $m_{ \nu, \mu } ( \xi ) := \phi (2^{-\mu} \xi) \phi (2^{-\nu} (\xi + 2^{\nu})) $, $\xi \in \mathbb{R}$ and
$$ (\sigma_{\nu, \mathcal{I}_j})_{\mu, \mathcal{I}_{j'}} (x) := \sum_{ \substack{J \in \mathcal{I}_{j'} : \\ |J|= 2^{- \mu}} }  \langle \sigma_{\nu, \mathcal{I}_j} \rangle_J \chi_J (x) .$$
Observe that since $ \mu < \nu - 1 $ one has
\begin{equation}\label{derivative_bound}
| ( m_{ \nu, \mu } )' (\xi ) | \leq 2^{-\mu} | \phi' (2^{-\mu} \xi) | \| \phi \|_{L^{\infty} (\mathbb{R})} + 2^{-\mu} | \phi (2^{-\mu} \xi) | \| \phi' \|_{L^{\infty} (\mathbb{R})}
\end{equation}
for all $\xi \in \mathbb{R}$. Hence, by using \eqref{derivative_bound} and an appropriate variant of \cite[Lemma 3.2]{Lerner}; see Section \ref{mod} below, one gets
\begin{equation}\label{variant_lemma}
\big\| T_{ m_{ \nu, \mu } } \big( P_{I_{\mu}^-} ( \widetilde{\psi}_{\nu, \mu} ) \big) \big\|_{L^2  ( (\sigma_{\nu, \mathcal{I}_j})_{\mu, \mathcal{I}_{j'}} ) }^2 \lesssim [\sigma]_{A_2}^2 \| \psi_{\nu, \mu} \|_{L^2 (\sigma)}^2,
\end{equation}
where we also used the fact that $| \widetilde{\psi}_{\nu, \mu} | = | \psi_{\nu, \mu} |$. 

By combining \eqref{equality} with \eqref{variant_lemma}, one establishes \eqref{bound_K_{0,1}} for $r'=0 $. One shows \eqref{bound_K_{0,1}} for $r'=1$ and \eqref{bound_K_{-1}} similarly. Notice that it follows from \eqref{pw_ineq}, \eqref{bound_K_{0,1}}, \eqref{bound_K_{-1}}, and \eqref{bound_nu} that \eqref{bound_j,0} holds.  We thus obtain \eqref{bound_j,r} for $r=0$. The proof of \eqref{bound_j,r} for $r=1$ is similar.   Therefore, the proof of \eqref{dual_main+} is now complete, in view of \eqref{rev_1}, \eqref{observation}, and \eqref{bound_j,r}. One shows \eqref{dual_main-} in an analogous way. We have thus established Theorem \ref{main}. 


\subsection{A variant of \cite[Lemma 3.2]{Lerner}}\label{mod}

In the previous section, inequality \eqref{variant_lemma} was obtained by using the following variant of \cite[Lemma 3.2]{Lerner}.

\begin{lemma}\label{modification}
Let $\mathcal{I}, \mathcal{J}$ be two given dyadic lattices in $\mathbb{R}$ and let $\mu, \nu \in \mathbb{Z}$ be such that $\mu < \nu$.  

If $m$ is bounded on an interval $K$ and differentiable in the interior of $K$ with
$$ C_{m,K} := \| m \|_{L^{\infty} (K)} + \int_{K} |m'(\xi)| d \xi < \infty, $$
then there exists an absolute constant $c_0 > 0$ such that
$$ \| P_{K} (T_m (f)) \|_{L^2 ( (\sigma_{\nu, \mathcal{I}})_{\mu, \mathcal{J}}) } \leq c_0 C_{m,K} [\sigma]_{A_2} (2^{-\mu} |K| + 1) \| f \|_{L^2 (\sigma)} $$
for every $A_2$ weight $\sigma$ on $\mathbb{R}$.
\end{lemma}

\begin{proof} 
Let $\mu, \nu$ be two given integers such that $\mu < \nu$ and let $\sigma$ be a given $A_2$ weight. 
We have
$$ (\sigma_{\nu, \mathcal{I}})_{\mu, \mathcal{J}} (x)  = \sum_{ \substack{J \in \mathcal{J} : \\ |J|= 2^{- \mu}} }  \langle \sigma_{\nu, \mathcal{I}} \rangle_J \chi_J (x) = \sum_{ \substack{J \in \mathcal{J} : \\ |J|= 2^{- \mu}} } \sum_{ \substack{I \in \mathcal{I} : \\ |I|= 2^{- \nu}} } \langle \sigma  \rangle_I \frac{|I \cap J|}{|J|} \chi_J (x) . $$
Fix a $J \in \mathcal{J}$ with $|J|= 2^{- \mu}$. Notice that, by arguing as in the proof of \cite[Lemma 3.2]{Lerner}, one has
\begin{equation}\label{M_pw}
| P_{K} (T_m (f)) (y) | \leq A_0 C_{m, K} T_K (f) (x) + \int_{K} H^{\ast} (M_{-t} f) (x) |m'(t)| dt
\end{equation}
for all $x,y \in 2J$, where $A_0 > 0$ is an absolute constant, $C_{m, K}$ is as in the statement of the lemma, and
$$ T_K (f) (x) := H^{\ast} (M_{-b} f) (x) + H^{\ast} (M_{-a} f) (x) + (2^{-\mu} |K| + 1) M (f)(x) .$$
Here $a$ and $b$ are the left and right endpoints of $K$, respectively. We write
$$ u(x) := A_0 C_{m, K} T_K (f) (x) + \int_{K} H^{\ast} (M_{-t} f) (x) |m'(t)| dt \quad \text{for}\  x \in 2J. $$
As in \cite{Lerner}, one has
$$ \frac{1}{|J|} \int_J | P_{K} (T_m (f)) (y) |^2 dy \leq \inf_{x \in 2J} [ u(x) ]^2 $$
and hence, for every $I \in \mathcal{I}$ with $| I | = 2^{-\nu}$ and $I \cap J \neq \emptyset$ one has
\begin{equation}\label{bound_I}
\langle \sigma \rangle_I \frac{|I|}{|J|} \int_J | P_{K} (T_m (f)) (y) |^2 dy \leq \int_I [u(x)]^2 \sigma (x) dx ,
\end{equation}
where we used the fact that if $I \cap J \neq \emptyset$ and $2 | I | \leq | J |$, then $I \subseteq 2J$. Hence, by using \eqref{bound_I} and the fact that the intervals $I \in \mathcal{I}$ with $I \subseteq 2J$ and $| I | = 2^{-\nu}$ have mutually disjoint interiors and their union is contained in $2J$, we deduce that
\begin{equation}\label{bound_J}
\sum_{ \substack{I \in \mathcal{I} : \\ |I|= 2^{-\nu}} } \langle \sigma \rangle_I \frac{|I \cap J|}{|J|} \int_J | P_{K} (T_m (f)) (y) |^2 dy \leq \int_{2J} [u(x)]^2 \sigma (x) dx 
\end{equation}
for every $J \in \mathcal{J}$ with $|J| = 2^{-\mu}$.

Since the intervals $J$ in $\mathcal{J}$ with $|J| = 2^{-\mu}$ `tile' $\mathbb{R}$, one deduces from \eqref{bound_J} that
\begin{equation}\label{final_bound}
\| P_{K} (T_m (f)) \|_{L^2 ( (\sigma_{\nu, \mathcal{I}})_{\mu, \mathcal{J}}) } \leq 2 \| u \|_{L^2 (\sigma)} .
\end{equation}
Arguing as in \cite{Lerner}, one completes the proof of the Lemma \ref{modification} by using \eqref{final_bound}, Minkowski's inequality, as well as the well-known bounds $ \| M \|_{L^2 (\sigma) \rightarrow L^2 (\sigma)} \lesssim [\sigma]_{A_2}$ and 
$ \| H^{\ast} \|_{L^2 (\sigma) \rightarrow L^2 (\sigma)} \lesssim [\sigma]_{A_2}$; see \cite{H-L} and \cite{H-P} for more general results involving two-weighted inequalities.
\end{proof}


\section{Proof of Theorem \ref{p-p}}\label{proof_1}

Arguing as in \cite[Remark 4.1]{Lerner}, the upper estimate in Theorem \ref{p-p} follows from Theorem \ref{main} combined with \cite[Theorem 3.1]{Duoandikoetxea}. 

The lower estimate in Theorem \ref{p-p} is obtained by adapting ideas from Bourgain's paper \cite{Bourgain_89} to our case. To be more specific, fix a Schwartz function $\eta$ satisfying the properties $\mathrm{supp} (\widehat{\eta}) \subseteq [1/2,4]$ and $\widehat{\eta}|_{[1,2]} \equiv 1$. For $N \in \mathbb{N}$, define $\eta_N$ by 
$$ \widehat{\eta_N} (\xi) := \widehat{\eta} (N^{-1} \xi), \quad \xi \in \mathbb{R} .$$
Notice that, since $\| \eta_N \|_{L^1 (\mathbb{R})} = \| \eta \|_{L^1 (\mathbb{R})}$ and $\| \eta_N \|_{L^2 (\mathbb{R})} = N^{1/2} \| \eta \|_{L^2 (\mathbb{R})}$, one deduces 
\begin{equation}\label{p_norm}
\| \eta_N \|_{L^p (\mathbb{R})} \lesssim N^{(p-1)/p} \quad (1 < p < 2). 
\end{equation}
Fix a $p \in (1,2)$ `close' to $1^+$ and choose $N \in \mathbb{N}$ such that 
\begin{equation}\label{choice_N,p}
\log N \sim (p-1)^{-1} .
\end{equation}
By using \eqref{p_norm} and \eqref{choice_N,p}, Minkowski's inequality, and H\"older's inequality, we have
\begin{align*}
\| S_{\mathcal{I}_{E_2}} \|_{L^p (\mathbb{R}) \rightarrow L^p (\mathbb{R})} \gtrsim \| S_{\mathcal{I}_{E_2}}  (\eta_N) \|_{L^p (\mathbb{R}) } & \geq \| S_{\mathcal{I}_{E_2}}  (\eta_N) \|_{L^p ([0,1]) } \\
&\geq \Bigg( \sum_{k \in \mathbb{Z}} \sum_{\substack{l \in \mathbb{Z}: \\ l < k}} \big\| P_{I_{k,l}^+} ( \eta_N ) \big\|_{L^p ([0,1])}^2 \Bigg)^{1/2} \\
&\geq \Bigg( \sum_{k=2}^{\lfloor \log N \rfloor} \sum_{l=1}^{k-1} \big\| P_{I_{k,l}^+} (\eta_N) \big\|_{L^1 ([0,1])}^2 \Bigg)^{1/2},
\end{align*}
where $\lfloor x \rfloor$ denotes the integer part of $x \in \mathbb{R}$. Since $\widehat{\eta_N} (\xi) = 1$ for $\xi \in [N, 2N]$, one can easily check that for all integers $k, l $ with $ 2 \leq k \leq \lfloor \log N \rfloor $ and $ 1 \leq l < k$ one has
$$ \big| P_{I_{k,l}^+} (\eta_N) (x) \big| \sim \Big| \frac{\sin (\pi l x)}{x} \Big| \quad \text{for all }  x \neq 0 .$$
Hence, a standard computation yields that
$$ \big\| P_{I_{k,l}^+} ( \eta_N ) \big\|_{L^1 ([0,1])} \sim \log \big( 4 \big| I_{k,l}^+ \big| \big) \sim l $$
for all $k, l \in \mathbb{N}$ with $ 2 \leq k \leq \lfloor \log N \rfloor $ and $ 1 \leq l < k$. We thus have
$$ \| S_{\mathcal{I}_{E_2}} \|_{L^p (\mathbb{R}) \rightarrow L^p (\mathbb{R})} \gtrsim \Bigg( \sum_{k=2}^{\lfloor \log N \rfloor} \sum_{l=1}^{k-1} l^2  \Bigg)^{1/2} \sim ( \log N )^2, $$
which, together with \eqref{choice_N,p}, shows that the lower estimate in Theorem \ref{p-p} holds true. 


\section{Some further remarks}\label{remarks}

\subsection{Periodic versions of Theorems \ref{p-p} and \ref{main}} If $f$ is a trigonometric polynomial on $\mathbb{T}$, define $S_2 (f)$ by
$$ S_2 (f) (\theta) := \Bigg( |\widehat{f}(0)|^2+ \sum_{ \substack{ k , l \in \mathbb{N}_0 : \\ k > l }} \Big[ |\Delta_{I^+_{k,l}} (f) (\theta)|^2 + |\Delta_{I^-_{k,l}} (f) (\theta )|^2 \Big] \Bigg)^{1/2} \quad (\theta \in \mathbb{T}) ,$$
where 
$$ \Delta_{I^+_{k,l}} (f) (\theta) : = \sum_{n \in I^+_{k,l}} \widehat{f} (n) e^{i 2 \pi n \theta}  \quad (\theta \in \mathbb{T})$$
and
$$  \Delta_{I^-_{k,l}} (f) (\theta) : = \sum_{n \in I^-_{k,l} } \widehat{f} (n) e^{i 2 \pi n \theta}  \quad (\theta \in \mathbb{T}) $$
with $I^+_{k,l}$ and $I^-_{k,l}$ being as in the Euclidean case; $I^+_{k,l} : = [2^k - 2^l, 2^k- 2^{l-1})$ and $I^-_{k,l} : = [ - 2^k + 2^{l-1}, - 2^k + 2^l )$. 

A straightforward adaptation of the argument of Section \ref{proof_2} to the periodic setting yields
\begin{equation}\label{main_T}
\| S_2 (f) \|_{L^2 (w) } \lesssim [w]^2_{A_2 (\mathbb{T})} \| f \|_{L^2 (w)}
\end{equation}
for all periodic $A_2$ weights $w$ and for every trigonometric polynomial $f$ on $\mathbb{T}$, where the implied constant in \eqref{main_T} is independent of $w$ and $f$. Recall that a non-negative integrable function $\sigma$ on $\mathbb{T}$ is said to be a periodic $A_2$ weight if, and only if,
$$ [\sigma]_{A_2 (\mathbb{T})} := \sup_{ \substack{ I \subseteq \mathbb{T} : \\ \text{arc}}}  \langle \sigma \rangle_I \langle \sigma ^{-1} \rangle_I < \infty  . $$
Moreover, by using \eqref{main_T} and a periodic version of \cite[Theorem 3.1]{Duoandikoetxea} as well as an adaptation of the argument of the previous section to the periodic setting, one deduces that there exist absolute constants $c_1, c_2 >0$ such that
\begin{equation}\label{p-p_T}
\frac{c_1} {(p-1)^2} \leq \| S_2 \|_{L^p (\mathbb{T}) \rightarrow L^p (\mathbb{T})} \leq \frac{c_2} {(p-1)^2} \quad (1 < p \leq 2).
\end{equation}


\subsection{Littlewood-Paley operators formed with respect to certain lacunary sets of finite order}
By arguing as in the proof of Theorem \ref{proof_2}, one can show that  $\| S_{\mathcal{I}_{\widetilde{E}_2}} \|_{L^2 (w) \rightarrow L^2 (w)} \lesssim [w]_{A_2}^2 $, where $ \widetilde{E}_2 := \{ \pm ( 2^k  +  2^l ) :  k , l \in \mathbb{Z}  \text{ with } k > l  \} $. More generally, if one considers the lacunary set of order $N \in \mathbb{N}$ (with $N \geq 2$) given by 
$$ \widetilde{E}_N : = \{ \pm ( 2^{k_1} + \cdots + 2^{k_N} ): k_1, \cdots , k_N \in \mathbb{Z} \text{ with }  k_1 > \cdots > k_N \} $$
and $S_{\mathcal{I}_{\widetilde{E}_N}}$ denotes the corresponding Littlewood-Paley operator, then by  suitably modifying and iterating the first part of the proof of Theorem \ref{proof_2} and next, by using an appropriate extension of Lemma \ref{modification}, one can show that
\begin{equation}\label{w_n}
 \| S_{\mathcal{I}_{\widetilde{E}_N}} \|_{L^2 (w) \rightarrow L^2 (w)} \lesssim [w]_{A_2}^{1 + N/2} 
\end{equation}
for any $A_2$ weight $w$ on $\mathbb{R}$, where the implied constant in \eqref{w_n} depends only on $N$; we omit the details. By using \eqref{w_n} and an adaptation of the argument of Section \ref{proof_1}, one gets
\begin{equation}\label{p-p_n}
 \| S_{\mathcal{I}_{\widetilde{E}_N}} \|_{L^p (\mathbb{R}) \rightarrow L^p (\mathbb{R})} \sim (p-1)^{-(1+N/2)} \quad (1<p \leq 2),
\end{equation}
where the implied constants in \eqref{p-p_n} depend only on $N$. Analogous results hold in the periodic setting.


\section*{Acknowledgements}

The author would like to thank Francesco Di Plinio for an interesting discussion during the HAPDE 2019 conference in Helsinki and in particular, for asking the author the question about the behaviour of the $L^p-L^p$ operator norm of $S_{\mathcal{I}_{E_2}}$ as $p \rightarrow 1^+$. The author would also like to thank Alan Sola for his useful comments that improved the presentation of this paper.

The author was supported by the `Wallenberg Mathematics Program 2018', grant no. KAW 2017.0425, financed by the Knut and Alice Wallenberg Foundation.


\end{document}